\documentclass[10pt]{article}
\usepackage{amsfonts,amssymb,amsthm,amsmath}
\usepackage{graphicx}

\theoremstyle{plain}
\newtheorem{lemma}{Lemma}[section]
\newtheorem{theorem}[lemma]{Theorem}

\theoremstyle{definition}
\newtheorem*{remark*}{Remark}

\newtheoremstyle{hypstyle}{}{}{}{}{\bfseries}{.}{ }%
{\thmname{#1}\thmnumber{H#2}\thmnote{}}

\theoremstyle{hypstyle}
\newtheorem{hypbase}{}
\newenvironment{hyp}{\pushQED{\qed}\hypbase}{\popQED\endhypbase}

\newcommand{\R}{{\mathbb R}}
\newcommand{\N}{{\mathbb N}}
\newcommand{\Z}{{\mathbb Z}}

\newcommand{\mA}{\mathcal{A}}
\newcommand{\mC}{\mathcal{C}}
\newcommand{\mD}{\mathcal{D}}
\newcommand{\mE}{\mathcal{E}}
\newcommand{\mL}{\mathcal{L}}
\newcommand{\mP}{\mathcal{P}}
\newcommand{\mS}{\mathcal{S}}
\newcommand{\mU}{\mathcal{U}}
\newcommand{\mW}{\mathcal{W}}

\renewcommand{\a}{\alpha}
\renewcommand{\b}{\beta}
\newcommand{\g}{\gamma}

\newcommand{\D}{\Delta}

\newcommand{\ph}{\varphi}
\newcommand{\lm}{\lambda}

\newcommand{\Om}{\Omega}
\newcommand{\p}{\pi}
\newcommand{\s}{\sigma}
\renewcommand{\t}{\tau}
\renewcommand{\th}{\vartheta}

\newcommand{\gr}{\nabla}

\newcommand{\br}{\boldsymbol{r}}
\newcommand{\bR}{\mathbf R}
\newcommand{\oxi}{\overline \xi}
\newcommand{\omm}{\overline m}
\newcommand{\oy}{\overline y}
\newcommand{\olm}{\overline \lm}

\newcommand{\intp}{\int_{0}^{2\p}}

\newcommand{\Ker}{\mathrm{Ker}\,}
\newcommand{\Range}{\mathrm{Range}\,}
\newcommand{\Span}{\mathrm{span}\,}

\numberwithin{equation}{section}

\begin{document}

\vspace{30mm}
\begin{center}
\begin{large}
\textbf{Bifurcation and Secondary Bifurcation of\\
Heavy Periodic Hydroelastic Travelling Waves}
\end{large}

\vspace{5mm}

\textsc{Pietro Baldi}\footnote{
\emph{E-mail}: \texttt{pietro.baldi@unina.it}}

\emph{Dipartimento di Matematica e Applicazioni ``R. Caccioppoli'',\\
Universit\`a di Napoli ``Federico II'',\\
Via Cintia, 80126 Napoli, Italy}

\textsc{and}

\textsc{John F. Toland}\footnote{
\emph{E-mail}: \texttt{jft@maths.bath.ac.uk}}

\emph{Department of Mathematical Sciences,\\
University of Bath,\\
Bath BA2 7AY, UK}
\end{center}

\vspace{1mm}

\begin{abstract}
The existence question for two-dimensional symmetric steady waves travelling on the surface of a deep ocean beneath a heavy elastic membrane is analyzed as a problem in bifurcation theory.
The behaviour of the two-dimensional cross-section of the membrane is modelled as a thin (unshearable), heavy, hyperelastic Cosserat rod, and the fluid beneath is supposed to be in steady two-dimensional irrotational motion under gravity.
When the wavelength has been normalized to be $2\pi$, and assuming that gravity and the density of the undeformed membrane are prescribed, there are two free parameters in the problem: the speed of the wave and drift velocity of the membrane.

It is observed that the problem, when linearized about uniform horizontal flow, has at most two independent solutions for any values of the parameters.
 When the linearized problem has only one normalized solution, it is shown that the full nonlinear problem has a sheet of solutions comprised of a family of curves bifurcating from simple eigenvalues. Here one of the problem's parameters is used to index a family of bifurcation problems in which the other is the bifurcation parameter.

When the linearized problem has two solutions,
with wave numbers $k$ and $l$
such that $\max\{k,l\} / \min\{k,l\} \notin \Z$,
it is shown that there are three two-dimensional sheets of bifurcating solutions. One consists of ``special'' solutions with minimal period $2\pi/k$; another consists of ``special'' solutions with minimal period $2\pi/l$; and the third, apart from those on the curves where it intersects the ``special'' sheets, consists of ``general'' solutions with minimal period $2\pi$.

 The two sheets of ``special'' solutions are rather similar to those that occur when the linearized problem has only one solution.
 However, points where the first sheet or the second sheet intersects the third sheet are period-multiplying (or symmetry-breaking) secondary bifurcation points on primary branches of ``special'' solutions. This phenomenon is analogous to that of Wilton ripples, which arises in the classical water-wave problem when the surface tension has special values. In the case of Wilton ripples, the coefficient of surface tension and the wave speed are the problem's two parameters. In the present context, there are two speed parameters, meaning that the membrane elasticity does not need to be highly specified for this symmetry-breaking phenomenon to occur.
%
\end{abstract}

\begin{small}
\emph{Keywords:} hydrodynamic waves, hydroelastic waves, nonlinear elasticity, free boundary problems, travelling waves, bifurcation theory, secondary bifurcations, Wilton ripples, Lyapunov-Schmidt reduction,
symmetry-breaking.

\emph{2000 Mathematics Subject Classification:}
35R35, 
74B20, 
74F10, 
76B07, 
37G40. 

\end{small}

\section{Introduction}
The existence question for symmetric, $2\pi$-periodic steady waves
travelling with speed $c_0$ on the surface of a heavy, inviscid fluid
which is at rest at infinite depth beneath a heavy, thin (unshearable) elastic membrane was considered in \cite{Toland-heavy} as a global problem in the calculus of variations.
Here we use local methods to study the bifurcation of such waves.
We will show that there are two free parameters, $c_0$ (the wave speed) and $d$ (the membrane drift velocity) and that, when the problem is linearized about a uniform horizontal stream with the membrane unstretched, there are no more than two linearly independent solutions. (See equation \eqref{pinotto}, in which $\lm_1 = \rho(c_0-d)^2$, $\lm_2= c_0^2$, $\rho$ is the density of the membrane and $\mC$ is the Hilbert transform of a $2\pi$-periodic function.)
When there is no non-zero solution, nonlinear waves do not bifurcate from uniform horizontal streams;
when there is only one solution, a sheet of solutions representing a parameterized family of bifurcations from simple eigenvalues occurs;
when there are two independent solutions, there bifurcate three sheets of small-amplitude periodic waves.
The latter corresponds to the presence of secondary bifurcations from curves of ``special'' solutions, the hydroelastic analogue of what are known as Wilton ripples \cite{wilton}, as described in the Abstract and in Section \ref{picture}. To quote from \cite{cd}, ``Waves characterized by two dominant modes are often called Wilton's
ripples in the literature in reference to Wilton's  paper (1915). It
turns out that the phenomenon described as Wilton's ripples was
accounted for at least twice prior to Wilton's  paper : in an
unpublished addendum to an essay that Bohr (1906) wrote in order to
win the Royal Danish Academy prize on the theme `The surface tension
of water' and in a paper by Harrison (1909).'' (Bohr's essay is \cite{bohr}.)

A key feature of the present analysis is the reduction of the physical problem (\ref{classical problem} a-g) to an equation \eqref{eq:F=0} for one $2\pi$-periodic function of one real variable and two parameters.

\subsection*{Physical Problem}
This problem is described in detail in \cite{Toland-heavy}. To summarise, we consider waves on the surface of an infinitely deep ocean beneath an elastic membrane under the assumption that there is no friction between the membrane and the fluid.
Since the water depth is infinite, there is no loss (after normalizing length scales appropriately) in restricting attention to waves that have period $2\pi$ in the horizontal direction.
The fluid's Eulerian velocity field is supposed to be two-dimensional and
stationary at the same time as the material of the membrane is in motion, driven by gravity, by forces and couples due to its elasticity and by pressure from the fluid.
The resulting mechanical behaviour of the surface membrane is modelled by regarding its cross-section as a heavy, unshearable, hyperelastic Cosserat rod, using the treatment in
Antman \cite[Ch.~4]{Antman}. We deal with this first.

\textbf{Membrane Elasticity.}
Let $(x,0) \in \R^2$ be the rest position of a material point in the membrane cross-section and let $\br(x) \in \R^2$ be its position after deformation. In the notation of Antman, $\th(x)$ is the angle between the horizontal positive semi-axis and the vector $\br'(x)$
(where $'$ means $d/dx$). Let
\begin{equation} \label{def nu mu}
\nu(x)=|\br'(x)| \, \text{ and } \,
\mu(x)= \th'(x).
\end{equation}
Thus $\nu(x)$ is the stretch of the membrane at the point $\br(x)$ and
\[
\widehat\s(\br(x)) = \frac{\mu(x)}{\nu(x)}
\]
is its curvature.
We suppose that the elastic properties of the membrane are described as follows.
\begin{hyp} \label{hyp:exists E}
\emph{(Hyperelasticity)} There exists a $C^\infty$-function $E(\nu,\mu) \geq 0$,~
$\nu > 0$, $\mu \in \R$,
such that, after the deformation $(x,0) \mapsto \br (x)$, the elastic energy in the reference segment $\{ (x,0) : x \in [x_1,x_2]\}$ is
\[
\mE (\br) = \int_{x_1}^{x_2} E(\nu(x), \mu(x)) \, dx,
\]
where $\nu(x), \mu(x)$ are defined in \eqref{def nu mu}. $E$ is called the stored energy function.
\end{hyp}
We also assume that the reference configuration, unstretched and unbent, is a local minimum of the elastic energy, which is locally convex.
\begin{hyp} \label{hyp:E=0 and convex}
\emph{(Rest state and local convexity)}
\[
E(1,0) = E_1(1,0) = E_2(1,0) = E_{12}(1,0)=0,\quad
E_{11}(1,0) > 0,~E_{22}(1,0)>0.
\]
Subscripts 1, 2 denote partial derivatives with respect to $\nu$, $\mu$, respectively.
\end{hyp}

\begin{remark*} Since this is a study of bifurcating waves, regularity questions that are significant for large-amplitude waves are unimportant here. It is therefore for convenience only that we suppose $E\in C^\infty$.
With a little more technical effort the theory can be developed for $E$ with much less regularity.
\qed\end{remark*}

\textbf{Travelling Waves.}
For a periodic travelling wave, the position at time $t$ of the material
point with Lagrangian coordinates $(x,0)$ in the undeformed membrane is assumed to be given by
\begin{equation*}
\bR(x ,t):\,=\big(x +dt +u(x -ct),\, v(x -ct)\big),
\end{equation*}
where $u$ and $v$ are $2\pi$-periodic and $c,\,d \in \R $. Let $c_0 = c+d$.
Then the surface profile at time $t$ is the curve
\begin{align*}
\mS_t &= \{(x +dt +u(x -ct),\, v(x -ct)) : x \in \R \} \\
&= \{(s+u(s),v(s)):s \in \R\}+(c_0t,0) \\
&=: \mS+(c_0t,0).
\end{align*}
Thus $\mS_t$ is represented by a profile $\mS$ of fixed shape propagating from left to right, say, without changing shape at a constant
velocity $c_0$ while at the same time the material point with
Lagrangian coordinates $(x ,0)$ has temporal period $2\pi/c$ relative to
a frame moving with speed $d$. We refer to $c_0$ as the \emph{wave speed} and to $d$
as the \emph{drift velocity of the membrane}, both
calculated relative to the fluid at rest at infinite depth.

Since the membrane is in motion relative to the moving frame, there are inertial effects due to its mass, but it is supposed throughout that there is no friction between the fluid and the membrane.
Under this assumption, it was shown in \cite{Toland-heavy} that the inertial effects lead to an equivalent steady-wave problem in which the wave speed and the drift velocity coincide, and the stored energy function is perturbed by a quadratic term.

\begin{remark*}
To motivate this observation and what follows, it is worth observing the corresponding situation that arises when travelling-wave solutions $R$ to an analogous nonlinear wave equation
\[
(E'(R_x))_x = \rho R_{tt}, \quad \rho>0,
\]
are sought.
This equation describes longitudinal motion in a one-dimensional elastic rod for which $E$ is the stored energy function, $\rho$ is the undeformed density and $R(x,t)\in \R$ is the position at time $t$ of the point with Lagrangian coordinate $x\in \R$. It is significantly simpler than the hydroelastic wave problem because here there are no body forces and because the shape of the rod does not change, only the relative positions of its material points in a straight line change with time.

First note that a {stationary ({time-independent}) solution} satisfies {$(E'(R_x))_x =0$} and is given by a critical point of the potential energy functional
\[
\int_0^{2\pi} E (R_x)\,dx.
\]
On the other hand, $R$ is a periodic travelling wave with drift velocity $d$ if
\[
R(x,t) =c_0t+r(x-ct),
\]
for some $c$ and $c_0$, where $r(s+2\pi) = 2\pi+r(s)$ and $d=c_0-c$.
Here we regard the variable $s$ $(\,= x-ct)$ as a steady Lagrangian coordinate for travelling waves.
The equation for $r(s)$ is then
\begin{equation} \label{ogue}
\{ E'(r_s) -c^2\rho r_s \}_s =0,
\end{equation}
which corresponds to critical point of
\[
\int_0^{2\pi} \Big(E (r_s)-\frac{\rho}2 c^2r_s^2 \Big)\,ds.
\]
Therefore periodic travelling waves correspond to a boundary-value problem for stationary solutions of a nonlinear wave equation with a \emph{different} stored energy function
\[
\underline E (p) := E(p) - \frac \rho 2 c^2p^2,
\]
instead of $E(p)$.
In \cite{Toland-heavy}, $\underline E$ is called the pseudo-potential energy of travelling waves.
 If $c$ is large $\underline E$ is not convex, even when $E$ is convex \cite{strain}.

 Note that if $r$ and $c$ correspond to a travelling wave, then a family, parametrized by $c_0\in \R$, of travelling waves with drift velocity $d$, is given by
\[
R(x,t)=c_0t+r(x-ct), \qquad d = c_0-c.
\]
However, in the hydroelastic wave problem, the interaction of the membrane with the fluid means that the dependence of waves on both parameters is not so trivial. In fact, we will see that the solutions $r(s)=s$,  $c_0$ and $c$ arbitrary, to this one-dimensional problem, when combined with a wave profile with zero elevation, is a family of trivial solutions of the hydroelastic wave problem which we describe next. They correspond to an undeformed membrane drifting with velocity $d=c_0-c$ on the surface of a uniform flow with horizontal velocity $c_0$.
\qed \end{remark*}

To summarise the analogous treatment in \cite{Toland-heavy} of the equivalent hydroelastic travelling-wave problem, let
$(s,0)$ be the steady Lagrangian coordinate, $\br(s)$ its deformed position
and let $\mu$ and $\nu$ be defined as in \eqref{def nu mu}, with $s$ in place of $x$.
 Let $\rho$ be the density of the \emph{undeformed} membrane section and
let $\eta (s)= \textbf{j} \cdot \br(s)$, where $\textbf{j}$ is the unit vector in the upward vertical direction ($\eta$ is the wave elevation).
\begin{subequations} \label{classical problem}
In \cite[eqn.~(1.8)]{Toland-heavy} it is shown that $\nu$, $\mu$ and $\eta$ satisfy
\begin{equation}
\frac{d~}{ds} \Big\{
\nu(s) \, E_1(\nu(s),\mu(s)) -\frac{\rho}{2}\, c^2 \nu(s)^2
+ \mu(s)\, E_2(\nu(s),\mu(s)) - E(\nu,\mu) -g \rho \eta(s)\Big\} = 0,
\end{equation}
which is the analogue of \eqref{ogue} in the present situation.
The pressure $P$ in the fluid, internal forces and gravity combine to deform the membrane. Thus, from \cite[(1.7e)]{Toland-heavy},
\begin{equation} \label{def P}
P(\br) = \frac{1}{\nu} \,
\Big( \frac{ E_2(\nu, \mu)_s}{\nu} \Big)_{\!s}
- \frac{\mu}{\nu} \big( E_1(\nu, \mu) -c^2\rho \nu \big) + \frac{g\rho \cos \vartheta}{\nu},
\end{equation}
where, as in \cite{Toland-heavy}, we assume that the material in one period of the membrane surface is a deformation of an interval of length $2\p$ of the reference membrane\footnote{The periodic wave does not require additional membrane mass as it passes underneath.}:
\begin{equation}  \label{r constrain}
\mS \cap \big([ 0, 2\p]\times \R\big)
= \{ \br(s) : s \in [s_0, s_0 + 2\p] \} \ \text{ for some} \ s_0 \in \R.
\end{equation}

\textbf{Fluid Motion.}
In the moving frame the fluid velocity field is two dimensional, stationary
and irrotational, the membrane cross-section coincides with a streamline and the flow at infinite depth is horizontal with velocity $-c_0$.
The surface $\mS$ is a zero level line for the stream function $\psi(X,Y)$, which is harmonic and represents horizontal laminar flow at infinite depth. So
\begin{align}
\label{D psi=0}
\D \psi = 0 & \quad \text{below } \mS, \\
\label{psi=0}
\psi = 0 & \quad \text{on } \mS
\ \text{(the kinematic boundary condition),} \\
\label{psi bottom}
\gr \psi (X,Y) \to (0,c_0) & \quad \text{as } Y \to - \infty.
\end{align}

\textbf{Membrane-fluid interaction.}
The dynamic boundary condition takes the form
\begin{equation} \label{psi dynamics}
-\frac12\, |\gr \psi(X,Y)|^2 - g Y +\frac{c_0^2}{2} = P(\br(s))
\quad \text{when } (X,Y) = \br(s) \in \mS.
\end{equation}
Here $P$ is given by \eqref{def P} and the left side of \eqref{psi dynamics} is the pressure in the fluid.
\end{subequations}

\subsection{The Free-Boundary Problem}\label{fbp}
 A steady hydroelastic wave with speed $c_0$ and drift velocity $d$ is a non-self-intersecting smooth $2\pi$-periodic curve $\mS$ in the plane for which there exists a solution of \eqref{classical problem} with $c = c_0-d$. Since we are interested in symmetric waves, we  require  that $\mS$ is symmetric about a vertical line.
To examine the solution set of this elaborate system we will use standard bifurcation-theory methods based on the implicit function theorem.
 Before going into the details, we give a schematic outline of what we have achieved (much remains to be done).
 Throughout we treat $g$ (gravity) and $\rho$ (density of the undeformed membrane) as constants and regard
\[
\lm_1 := c^2 \rho, \quad \lm_2 := c_0^2, \quad \lm=(\lm_1,\lm_2)\in \R^2,
\]
as the physical parameters. (Because bifurcation theory studies the existence of solutions with small amplitudes and slopes, the question of self-intersection of wave surfaces does not arise in the present study. By contrast, in the theory of large-amplitude waves \cite{Baldi-Toland-var}, considerable effort is required to ensure that there is no self-intersection.)

\subsection{Bifurcation Picture} \label{picture}
In this section we explain schematically the geometric picture of small amplitude hydroelastic waves close to a bifurcation point.
See Theorems \ref{thm:simple bif} and \ref{thm:secondary bif} for a detailed statement.

The first observation is that, for all choices of the two independent parameters, the problem, when linearized at the trivial solution of uniform horizontal flow under an undeformed membrane, has at most two linearly independent solutions. If there is only one linearized solution when $(\lm_1,\lm_2) = (\lm_1^*,\lm_2^*)$ say, there is at most only one linearized solution for all nearby $(\lm_1,\lm_2)$.
Therefore, with either one of the parameters held fixed, there are bifurcations from simple eigenvalues with respect to the other parameter. Their union is a two-dimensional sheet of solutions which bifurcates from $(\lm_1^*,\lm_2^*)$. The details are in Section \ref{simple}.

On the other hand, suppose that at $\lm^*= (\lm_1^*,\lm_2^*)$ the linearized problem has two solutions $\cos (k\t)$ and $\cos (l\t)$, where $k$ and $l$ are positive integers
with $\max\{k,l\} / \min\{k,l\} \notin \Z$.
By restricting attention to solutions in $Z_k=\Span \{\cos (jk\t):j \in \N\}$, or in $Z_l=\Span \{\cos (jl\t): j \in \N\}$, the problem may be reduced to one of ``bifurcation from a simple eigenvalue'' for particular solutions that have \emph{minimal period} $2\pi/k$ or $2\pi/l$, respectively.
This is straightforward and similar to what is done in Section \ref{simple}.
Locally we obtain a sheet of solutions of minimal period $2\pi/k$ and a sheet of solutions of minimal period $2\pi/l$.
Each of these sheets is locally the graph of a function which gives $\lm_2$ in terms of the wave amplitude and $\lm_1$ (the roles of $\lm_2$ and $\lm_1$ can be reversed).
We will refer to these solutions with minimal periods less than $2\pi$ as ``special'' solutions.

 In Section \ref{double} we show that in addition to these two sheets of ``special'' solutions there is a two-dimensional sheet of ``general'' solutions and that the sheet of ``general'' solutions intersects each of the sheets of ``special'' solutions in a curve.
The solutions on the sheet of ``general'' solutions, except
where it intersects the sheet of ``special'' solutions,
have minimal period $2\pi$. Therefore the general solutions on this sheet represent a symmetry-breaking (or period-multiplying) secondary bifurcation on the curves of special solutions.
 This is the hydroelastic analogue of Wilton ripples, a type of water wave which arises in the presence of surface tension \cite{JT,TJ,wilton}. In Wilton-ripple theory there are also two parameters, the wave speed and the surface tension coefficient (which measures surface elasticity). Wilton ripples bifurcate from uniform streams at certain values of the wave speed, when the surface tension has particular values.
 Here the two parameters are independent of the elasticity of the membrane. Therefore the wave and drift speeds can conspire to produce ripples, for \emph{any} prescribed elastic membrane.

After Lyapunov-Schmidt reduction, when the linearized problem has two independent solutions, $\cos(k\t)$ and $\cos(l\t)$, there are two bifurcation equations in four unknowns
\begin{equation}\label{qual}
\Phi_k (t_1,t_2,\lm_1,\lm_2) = 0, \quad
\Phi_l (t_1,t_2,\lm_1,\lm_2) = 0,
\end{equation}
where $t_1$ and $t_2$ near 0 are the coefficients of $\cos (k\t)$ and $\cos (l\t)$, respectively.
Solutions with $t_1=0$ or $t_2=0$ correspond to ``special'' solutions.
The hypothesis that $\max\{k,l\} / \min\{k,l\} \notin \Z$
leads to the key observation that, for all $t_1,t_2$ near 0,
\[
\Phi_k (0,t_2,\lm_1,\lm_2) = 0, \quad
\Phi_l (t_1,0,\lm_1,\lm_2) = 0.
\]
Then each of the sheets of ``special'' solutions is found by solving one equations in three unknowns, as is done in Section \ref{simple}.

To find the ``general'' solutions, we seek solutions of \eqref{qual} with neither $t_1$ nor $t_2$ equal to 0.
This problem is reduced in Section \ref{double} to a desingularized one of the form
\begin{equation} \label{help}
\Psi_k (t_1,t_2,\lm_1,\lm_2) = 0, \quad
\Psi_l (t_1,t_2,\lm_1,\lm_2) = 0,
\end{equation}
for which $(0,0,\lm^*)$ is a solution
and $\partial (\Psi_k,\,\Psi_l)/\partial (\lm_1,\,\lm_2)$ at $(0,0,\lm^*)$ is invertible.
Then the implicit function theorem gives $\lm$ in a neighbourhood of $\lm^*$ in $\R^2$ as a function of $(t_1,t_2)$ in a neighbourhood of the origin in $\R^2$.
This is the sheet of general solutions we are seeking. It intersects the ``special'' sheets when $t_1=0$ or $t_2=0$.
The present analysis, which yields a \emph{qualitative local description} of the set of all $2\pi$-periodic hydroelastic waves near a bifurcation point, is sufficient to show that secondary bifurcations occur.

\begin{figure}[ht]
\quad
\includegraphics[scale=0.75]
{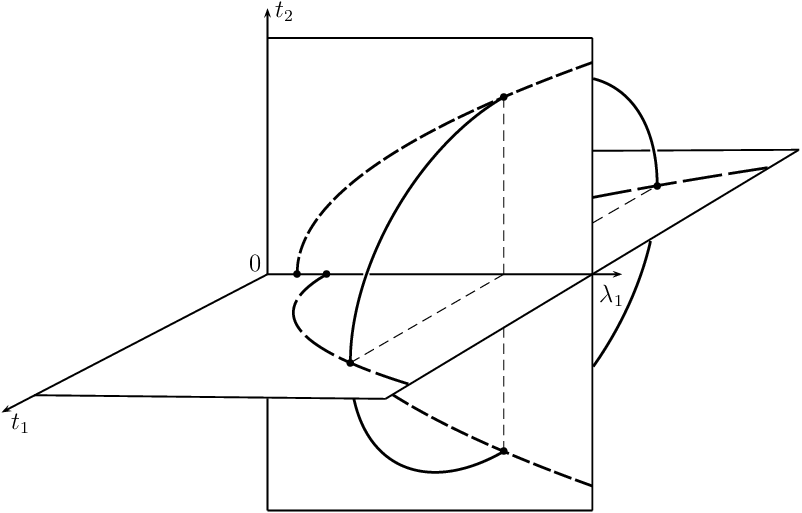}
\caption{A possible bifurcation diagram in the space $(t_1,t_2,\lm_1)$, when $\lm_2$ is fixed. The dashed curves correspond to the two branches of ``special'' solutions on the planes $t_1=0$ and $t_2=0$, and the solid curve gives the secondary branch of ``general'', symmetry-breaking solutions.}
\label{fig 1}
\end{figure}

\begin{remark*}
A more detailed geometrical description of the bifurcating sheets depends on the coefficients in Taylor series arising in the bifurcation equations,
and to calculate their values can be very complicated.
For example, suppose that we want to draw a picture of the solution set when one of the parameters, $\lm_2$ say, is fixed, and assume that \eqref{help} has the form
\begin{align*}
f(\lm_1) + (At_1 ^2+Bt_1t_2 +Ct_2^2) -\lm_2 &=0,\\
g(\lm_1) + (\a t_1 ^2+\b t_1t_2 +\g t_2^2) -\lm_2 &=0,
\end{align*}
where $f$ and $g$ are smooth functions and $A,\,\a,\,B,\,\b,\,C,\,\g$ are constants.
Suppose also that
\[
f'(\lm_1^*) \neq g'(\lm_1^*).
\]
Then it is clear from the implicit function theorem that locally
\[
\lm_1 = \Lambda \big( (A-\a)t_1^2 + (B-\b)t_1t_2 + (C-\g)t_2^2 \big),
\]
for some function $\Lambda$ with $\Lambda'(0)\neq 0$.
Therefore, for fixed $\lm_2$ close to $\lm_2^*$
the solution set is given locally by the level set of the function
\[
f \big( \Lambda \big( (A-\a)t_1^2 + (B-\b)t_1t_2 + (C-\g)t_2^2 \big) \big)
+ At_1^2 + Bt_1t_2 + Ct_2^2.
\]
The complexity of the dependence of this set on $A,\,\a,\,B,\,\b,C,\,\g,\,f$ and $g$ is evident.
These quantities in turn depend, in an explicit but highly non-trivial way, on the elastic properties of the membrane.
\qed\end{remark*}

\begin{remark*} In the case when $\max\{k,l\} / \min\{k,l\} \in \Z$, analysis is still possible, but the details are yet more complicated.
\qed\end{remark*}

\begin{remark*}
In the present work solutions will be found for values of parameters that are not covered by the maximization argument in \cite{Toland-heavy}.
Indeed, a convexity hypothesis of the form $E_{11}(\nu,\mu)>c^2\rho=\lm_1$ was crucial in the existence proof of \cite{Toland-heavy},
whereas both primary and secondary bifurcations occur here provided only that $\lm_1 \neq E_{11}$, see Lemmas \ref{lemma:exists one} and \ref{lemma:kl}.
\qed\end{remark*}

\section{Mathematical formulation}
Suppose that in the free-boundary problem \eqref{classical problem}, the \emph{shape of $\mS$} is known. Then $\psi$ is given by the unique solution to (\ref{classical problem}\,d,e,f).
Thus the kinetic and potential energies of the fluid are determined solely by the shape of $\mS$.
On the other hand, the elastic and gravitational potential energy of the membrane are determined by the positions of the material points $(x,0)$ in the deformed membrane. To deal with this distinction, and ultimately to prove in Section \ref{further} that only the shape matters, suppose that the shape of $\mS$ is given by a parametrization
\[
\mS=\{\varrho(\t): \t \in \R\}, \ \ \text{where} \
\varrho (\t+2\pi) = (2\pi,0) + \varrho(\t).
\]
Then, as in \cite{Toland-heavy}, we seek $\bR$ for travelling waves in the form
\begin{equation}\label{chii}
\bR(x,t):\,= (c_0 t,0) + \varrho(\chi(x -ct)),
\end{equation}
where $\chi:\R \to \R$ is a diffeomorphism with
$\chi (s+2\pi) = 2\pi + \chi(s)$.
As before, $s$ is the steady travelling-wave Lagrangian coordinate,
the unknowns are $\varrho(\t)$ and $\chi(s)$,
and $\br(s) = \varrho(\chi(s))$, so that $\mS = \{\br(s):s \in \R\}$.
To develop this approach, we recall the class of parametrizations in \cite{Toland-heavy}, specially tailored for this problem.

Let $w$ be a $2\p$-periodic real-valued function with second derivative locally square-integrable on $\R$, and let $\mC$ denote its Hilbert transform. Then
\begin{equation}\label{rho}
\varrho(w)(\t) := (-\t-\mC w(\t), \, w(\t)),
\quad \t \in \R,
\end{equation}
is a $2\pi$-periodic curve in the plane.
Thus $w$ is the unknown that describes the wave shape.
The other unknown is the stretch of the reference membrane.
To describe it we follow \cite{Baldi-Toland-var} by introducing diffeomorphisms $\kappa(\t)$ ($\kappa = \chi^{-1}$ in \eqref{chii}) of the interval $[0,2\p]$ such that $\chi(0) = 0$ and $\chi(2\p)=2\p$.
Then if the material point $s$ of the membrane in the reference configuration is
\begin{equation} \label{x=chi(tau)}
s = \kappa(\t)
\quad
\forall\, \t \in \R,
\end{equation}
its position $\br(s)$ after deformation is
\[
\br(s) = \varrho(w)(\t),
\]
and the stretch of the membrane is given in terms of $\kappa$ and $w$ by
\[
\nu(s) = \frac{|\varrho(w)'(\t)|}{\kappa'(\t)} =
\frac{\Om(w)(\t)}{\kappa'(\t)} \,,
\]
where
\[
\Om (w)(\tau) = \sqrt{w'(\tau)^2+ (1+\mC w'(\tau))^2}.
\]
Also, since $\br(s) = \varrho(w)(\t)$, there is a useful formula for the curvature,
\[
\widehat\s(\br(s)) = \sigma (w)(\tau) =
- \frac{1}{\Om(w)(\t)}\,\, \mC\Big(\frac{ \Om (w)'(\tau)}{\Om(w)(\tau)}
\Big).
\]
Roughly speaking, given a profile parametrized by \eqref{rho}, the diffeomorphisms \eqref{x=chi(tau)} describe the family of all \emph{physical deformations $(s,0) \mapsto \br(s)$ of the reference state} which produce the same profile $\varrho(w)$.
Note that the parametrization of a curve and the stretch of the material in it are independent.
It will be convenient later to replace $\kappa$ with the new unknown
\[
\xi(\t) := \kappa'(\t) - 1.
\]
Then $\xi$ has zero mean, $$ \frac{1}{2\pi}\intp \xi(\tau) \,d\tau=0, $$ because $\kappa(0) = 0$ and $\kappa(2\p)=2\p$.
In \cite{Toland-heavy} it is shown how solutions to this hydroelastic wave problem with a heavy membrane corresponds to critical points of a Lagrangian which, in terms of the unknown $(w,\xi)$, is written
\begin{align} \label{J}
J(w,\xi) := {} &
 \frac{c_0^2}{2} \intp w \mC w' \, d\t \,
- \frac{g}{2} \intp w^2 (1 + \mC w') \, d\t \\
& - \intp (1 + \xi) \, E \Big( \frac{\Om(w)}{1+\xi}\,, \frac{\Om(w)\, \s(w)}{1+\xi} \Big) \, d\t \,
+ \frac{c^2 \rho}{2} \intp \frac{\Om(w)^2}{1+\xi}\,d\t \notag \\
& - g \rho \intp w (1+\xi) \, d\t. \notag
\end{align}
The first term in the right in \eqref{J} is the fluid's kinetic energy in one period, relative to the moving frame;
the second term, with a minus, is the change in gravitational potential energy of the same body of fluid relative to a uniform flow;
the third is minus the elastic energy of one period of the deformed membrane;
the fourth is plus the kinetic energy of the membrane;
the fifth term is minus the gravitational potential energy of one period of the membrane. (For the derivation, see the discussion leading to \cite[(2.18)]{Toland-heavy}.)

This will be the starting point for this analysis.
Local bifurcation theory will be used to give a complete descriptions of all small-amplitude $2\pi$-periodic waves represented by critical points of $J$ close to a bifurcation point.

\section{The equations}
Define the notation $(E - \gr E \cdot)(\nu,\mu)$ by
\[
(E - \gr E \cdot)(\nu,\mu) :=
E (\nu,\mu) - \nu\, E_1 (\nu,\mu) - \mu\, E_2 (\nu,\mu).
\]
Suppose that $w$ and $\xi$ are small in an appropriate norm and that $J$ is differentiable at $(w,\xi)$.
Then the partial derivative with respect to $\xi$ in a direction $\eta$, where $[\eta]=0$, is
\[
d_\xi J(w,\xi)\, \eta
= - \intp \eta \, \Big\{ (E - \gr E \cdot)\Big( \frac{\Om(w)}{1+\xi}\,, \frac{\Om(w) \s(w)}{1+\xi} \Big) \, \,
+ \frac{c^2 \rho}{2}\, \Big( \frac{\Om(w)}{1+\xi}\Big)^2 + \, g \rho w \Big\} \, d\t.
\]
Let $\mL[w]$, depending on $w$, be the linear operator defined in \cite[Section 4.5]{Baldi-Toland-var} by
\[
\mL[w](u) := \frac{ w' u + (1+\mC w')\, \mC u}{\Om(w)^2}\,,
\]
with the property that
\begin{equation} \label{dirs}
d_w \Om(w) h = \Om(w) \, \mL[w](h'),
\quad
d_w \big(\Om(w) \s(w) \big) h = - \mC \big(\mL[w](h')\big)'.
\end{equation}
Then the partial derivative of $J(w,\xi)$ with respect to $w$ in the direction $h$ is
\begin{align*}
d_w J(w,\xi) h \, &=
\intp h \, \Big\{ c_0^2 \, \mC w' - g w (1+\mC w') - g \, \mC(w w') - g \rho (1+\xi) \Big\}\, d\t \\
 &\quad + \intp \Big\{ \frac{c^2 \rho \,\Om(w)}{1+\xi}\, - E_1\Big( \frac{\Om(w)}{1+\xi}\,, \frac{\Om(w) \s(w)}{1+\xi} \Big) \Big\} \, \Om(w) \, \mL[w](h') \, d\t \,
\notag\\&\quad + \intp E_2\Big( \frac{\Om(w)}{1+\xi}\,, \frac{\Om(w) \s(w)}{1+\xi} \Big) \, \big( \mC \mL[w](h') \big)' \, d\t.
\end{align*}
Now, suppose that $(w,\xi)$ is a critical point of $J$.
Define the projection
\[
\mP u := u - \frac{1}{2\pi}\intp u(\tau) \,d\tau,
\]
for all $2\p$-periodic, locally integrable functions $u$.
Then $\mP u$ has zero mean on $[0,2\pi]$.
Note that the operator $\mL[w]$ is independent of the mean of $w$.
A simple calculation shows that $d_w J(w,\xi)=0$ if and only if
\[
\frac{1}{2\pi} \intp (w + w\,\mC w')\,d\tau + \rho = 0
\quad \text{and} \quad
d_w J(w,\xi) \mP h = 0 \quad \forall\, h.
\]
Hence it suffices to study the equation $dJ_0(w,\xi) =0$, where $J_0$ is defined by
\begin{gather*}
J_0 (w,\xi) := J(w,\xi) + g\p\,\Big( \frac{1}{2\pi} \intp w \mC w'\,d\t + \rho \Big)^2,
\\
\intertext{over a class of functions satisfying}
\intp w(\tau) \,d\tau=\intp \xi(\tau) \,d\tau=0,
\end{gather*}
because then $dJ_0(w,\xi)=0$ implies that $dJ(w^*,\xi)=0$, where
\[
w^* := w - \frac{1}{2\pi}\intp w\,\mC w'\,d\t -\rho.
\]
From a calculation similar to that for $J$, partial derivatives of $J_0$ are given by
\begin{equation} \label{dxi J0}
d_\xi J_0(w,\xi) \,\eta = 
 - \intp \eta \, \Big\{ (E - \gr E \cdot)\Big( \frac{\Om(w)}{1+\xi}\,, \frac{\Om(w) \s(w)}{1+\xi} \Big) \, \,
+ \frac{c^2 \rho}{2}\, \Big( \frac{\Om(w)}{1+\xi}\Big)^2 + \, g \rho w \Big\} \, d\t
\end{equation}
and
\begin{align} \notag
d_w J_0 (w,\xi) h\, = {} &
\intp h \, \Big\{ c_0^2 \, \mC w' - g w (1+\mC w') - g \, \mC(w w') - g \rho (1+\xi) \Big\}\, d\t
\notag \\\notag
& + \intp \Big( c^2 \rho \, \frac{\Om(w)}{1+\xi}\, - E_1\Big( \frac{\Om(w)}{1+\xi}\,, \frac{\Om(w) \s(w)}{1+\xi} \Big) \Big) \, \Om(w) \, \mL[w](h') \, d\t \,
\\&+ \intp E_2\Big( \frac{\Om(w)}{1+\xi}\,, \frac{\Om(w) \s(w)}{1+\xi} \Big) \, \mC (\mL[w](h'))' \, d\t
\notag \\
& + \, 2g \, \Big( \frac{1}{2\pi} \intp w\mC w'\,d\t + \rho \Big) \, \intp h\, \mC w' \, d\t.
\label{dw J0}
\end{align}
In \cite[Sections 4.1 \& 4.2]{Baldi-Toland-var} the membrane density $\rho$ is zero. However, the calculations there are easily extended to take account of the extra terms here which involve $\rho > 0$.
To proceed, we adapt the notation from \cite{Baldi-Toland-var} to the case of positive $\rho$.
 For $u$ with zero mean, let
\[
\gr I_0 := \Big( c_0^2 + \frac g\pi\intp w \mC w'\,d\t + 2g\rho \Big) \, \mC w' - g w (1+\mC w') -g \, \mC (w w') - g \rho (1+\xi).
\]
With the $L^2$-adjoint of the inverse operator $\mL[w]^{-1}$ given by
\[
(\mL[w]^{-1})^* (u) = w' u + \mC ((1+\mC w')u),
\]
let
\begin{equation} \label{m0}
m_0 := (\mL[w]^{-1})^* \Big( \int_0^\t \mP(\gr I_0) \, dt \Big).
\end{equation}
Then, as in \cite{Baldi-Toland-var},
the equations for critical points of $J_0$ can be written as follows:
\begin{subequations} \label{P Euler both}
\begin{equation}
\label{P Euler xi}
\mP \Big\{ (E - \gr E \cdot)\Big( \frac{\Om(w)}{1+\xi}\,, \frac{\Om(w) \s(w)}{1+\xi} \Big) \, + \frac{c^2 \rho}{2}\,
\Big( \frac{\Om(w)}{1+\xi}\Big)^2 \Big\} + g \rho w = 0,
\end{equation}
\begin{multline}
\label{P Euler w}
\mP E_2\Big( \frac{\Om(w)}{1+\xi}\,,
\frac{\Om(w) \s(w)}{1+\xi} \Big)
+ \mC \Big\{ \int_0^\t \mP \Big( m_0 + \Om(w)
E_1\Big( \frac{\Om(w)}{1+\xi}\,, \frac{\Om(w) \s(w)}{1+\xi} \Big) \\
- c^2 \rho\, \frac{\Om(w)^2}{1+\xi}\Big) \, dt \Big\} = 0.
\end{multline}
\end{subequations}
Note that $(w,\xi)=(0,0)$ solves \eqref{P Euler both} for all values of $c,c_0,g,\rho$.

The free-boundary problem in Section \ref{fbp} is for symmetric waves, so we can simplify matters  by studying the bifurcation problem in spaces of even functions. This is what we do subsequently. Note that, if $w$ and $\xi$ are even functions,
then $\mC w'$, $\Om(w)$, $\s(w)$ and $1+\xi$ are also even.
\begin{lemma} \label{lemma:even}
Suppose that $(w,\xi)$ are even functions such that
\begin{equation} \label{critical pt}
d J_0 (w,\xi)\,(h,\eta) = 0
\end{equation}
for all even $(h,\eta)$ with sufficient regularity.
Then \eqref{critical pt} holds for all $(h,\eta)$.
\end{lemma}

\begin{proof}
Since every $2\p$-periodic function is the sum of even and odd functions, by the hypothesis it suffices to observe that
\eqref{critical pt} holds for $(h,\eta)$ odd when $w,\,\xi$ are even.
This follows by \eqref{dxi J0} and \eqref{dw J0}, since $\mL[w](h')$ is odd for odd $h$.
\end{proof}

\section{A further simplification} \label{further}
In this section we use the implicit function theorem to show that, for solutions of \eqref{P Euler xi}, $\xi$ is a function of $(w,\lambda_1)$ near $w=\xi=0$. This means that the stretch variable $\xi$ can be eliminated and the problem reduced to one for the unknown shape which is given by $w$.
For $k \in \N$, let $H^k_0$ denote the space of real-valued, $2\p$-periodic, even, zero-mean functions, with $k$th weak derivative locally square-integrable. For $r>0$, let $B_r(X)$ denote the open ball of radius $r$ centred at the origin in a Banach space $X$.
Fix $r>0$ such that
\[
\frac12\, \leq \Om(w) \leq 2,
\quad
|\s(w)| \leq 1,
\quad
|\xi| \leq \frac12
\]
for all $(w,\xi)\in
B_r(H^3_0) \times B_r(H^1_0)$.
Then a map $M \colon B_r(H^3_0) \times B_r(H^1_0) \times (0,+\infty) \to H^1_0$ may be defined by
\[
M(w,\xi,\lm_1) := \mP \Big\{ (E - \gr E \cdot)\Big( \frac{\Om(w)}{1+\xi}\,, \frac{\Om(w) \s(w)}{1+\xi} \Big) \, + \frac{\lm_1}{2}\,
\Big( \frac{\Om(w)}{1+\xi}\Big)^2 \Big\} + g \rho w,
\]
because $E$ is smooth, and $\Om(w)/(1+\xi)$ and $\Om(w) \s(w) / (1+\xi)$ are bounded functions.
Thus, the Euler equation \eqref{P Euler xi} may be written as
\[
M(w,\xi,\lm_1) = 0.
\]
Moreover, $M(0,0,\lm_1)=0$ for all $\lm_1$.
We note that
\[
M \in C^\infty \big( B_r(H^3_0) \times B_r(H^1_0) \times (0,+\infty), \, H^1_0\big),
\]
and, when $(w,\xi) = (0,0)$,
\begin{equation} \label{facts at 0}
\Om(0) \equiv 1, \quad \s(0) \equiv 0, \quad
d\Om(0)h = \mC h', \quad d(\Om \s)(0) h = h''.
\end{equation}
Since $ \mL[0] = (\mL[0]^{-1})^* = \mC$, it follows from H\ref{hyp:E=0 and convex} and \eqref{dirs} that
\begin{align}
& d_w M(0,0,\lm_1)\,h = - (E_{11} - \lm_1) \, \mC h' + g \rho h,
\label{dw M} \\
& d_\xi M(0,0,\lm_1)\, \eta = (E_{11} - \lm_1) \eta,
\label{dxi M} \\
& d_{\lm_1} M(0,0,\lm_1) = 0, \label{dc M}
\end{align}
where for convenience we have written $E_{11}$ instead of $E_{11}(1,0)$.

\begin{lemma} \label{lemma:IFT xi(w,c)}
When $\widehat \lm_1 \neq E_{11}$,
there exist a neighbourhood $\mW$ of $(0,\widehat\lm_1)$ in $H^3_0 \times (0,+\infty)$ and a map $\oxi \in C^\infty(\mW, B_r(H^1_0))$ such that
\[
M(w,\oxi(w,\lm_1),\lm_1) = 0
\]
for all $(w,\lm_1) \in \mW$, and, if $M(w,\xi,\lm_1) = 0$ with $(w,\lm_1) \in \mW$, then $\xi=\oxi(w,\lm_1)$.
Moreover, for all $\lm_1 \neq E_{11}$,
\begin{align}
& \oxi(0,\lm_1) = 0, \label{xi(0,c)=0} \\
& d_w \oxi(0,\lm_1)\,h = \,
\mC h' + \frac{g\rho}{\lm_1 - E_{11}}\, h. \label{dw xi}
\end{align}
\end{lemma}

\begin{proof} To obtain the existence of $\overline \xi$, apply the implicit function theorem using \eqref{dxi M}.
The required formulae then follow by the chain rule and \eqref{dw M} and \eqref{dc M}.
\end{proof}

Because of this, all solutions $(w,\xi,\lm)$ of \eqref{P Euler both} with $(w,\lambda_1) \in \mW$ are solutions of \eqref{P Euler w} of the form $(w,\oxi(w,\lm_1),\lm)$.
Let $\omm (w,\lm)$ denote $m_0$ (see \eqref{m0})
when $\xi = \oxi(w,\lm_1)$, and let
\[
e_i(w,\lm_1) := E_i \Big( \frac{\Om(w)}{1+\oxi(w,\lm_1)}\,,\, \frac{\Om(w)\s(w)}{1+\oxi(w,\lm_1)} \Big), \quad i = 1,2.
\]
Then on the set $\mD:= \{(w,\lm): (w,\lm_1) \in \mW,~ \lm_2 > 0 \}$, define the function
\[
F(w,\lm) := \,\mP e_2(w,\lm_1) + \mC \Big\{ \int_0^\t
\mP \Big( \omm(w,\lm) + \Om(w)e_1(w,\lm_1)- \frac{\lm_1\, \Om(w)^2}{1+\oxi(w,\lm_1)}\Big)\, dt \Big\},
\]
so that the system \eqref{P Euler both}, for $(w,\lambda_1) \in \mW$ and $\lm_1 \neq E_{11}$, becomes
\begin{equation} \label{eq:F=0}
F(w,\lm) = 0,\quad (w,\lm) \in \mD.
\end{equation}

\section{The linearized equation} \label{sec:linearized}
Recall that $F(0,\lm)=0$ for all $\lm$ with $\lm_1 \neq E_{11}$ and that $F \in C^\infty(\mD,\,H^1_0)$. We now calculate its partial derivatives.
When $w=0$, \eqref{facts at 0} holds, $\oxi=0$ by \eqref{xi(0,c)=0}
and, as a consequence, $\mP(\gr I_0)=0$.
Hence, by \eqref{dw xi},
\[
d_w \omm(0,\lm)\,h = -(\lm_2 + g\rho)\,h +
\Big( \frac{(g\rho)^2}{E_{11}-\lm_1}\, - g \Big) \, \mC \Big( \int_0^\t h(t)\,dt \Big),
\]
and, by H2,
\begin{align} \label{dF}
d_w F(0,\lm)\,h = & {}
\ E_{22} h''
+ \lm_1\, h
- \lm_2 \, \mC \Big( \int_0^\t h(t)\,dt \Big)
\\
& + \Big( g + \frac{(g\rho)^2}{\lm_1-E_{11}} \Big)\,
\mP \Big( \int_0^\t \mP \Big(\int_0^t h(s)\,ds \Big) \, dt \Big),
\notag
\end{align}
where, as above, $E_{ii} = E_{ii}(1,0)$, $i=1,\,2$.

Now suppose that $h \in H^3_0\setminus \{0\}$ and that $d_w F(0,\lm)h=0$.
Then $h \in H^5_0$, and differentiating
the equality $d_w F(0,\lm)\,h=0$ twice with respect to $\t$ yields
\begin{equation} \label{pinotto}
E_{22} h'''' + \lm_1 h'' - \lm_2 \,\mC h'
+ \Big( g + \frac{(g\rho)^2}{\lm_1-E_{11}} \Big) \, h = 0.
\end{equation}
It is easy to see that \eqref{pinotto} has a non-constant even solution $h$ if and only if
\begin{equation} \label{integer eq}
E_{22} k^4 - \lm_1 k^2 - \lm_2 k + g
+ \frac{(g\rho)^2}{\lm_1-E_{11}}\, = 0,
\end{equation}
for some positive integer $k$.
Since $E_{22}(1,0)$ is assumed to be positive, for every fixed $\lm_1,\lm_2 > 0$, \eqref{integer eq} possesses at most two positive integer solutions. (This follows by noting that the graph of the quartic curve $x \mapsto E_{22} x^4 - \lm_1 x^2$ on the half-plane $\{ x \geq 1\}$ intersects any straight line with slope $\lm_2$ at most twice.)

\subsection{Non-trivial kernel} \label{ntk}
Let $g,\rho, E_{11}, E_{22}>0$ be fixed.
Then $\lm_1,\lm_2>0$, $\lm_1 - E_{11}\neq 0$ and \eqref{integer eq} holds for some integer $k \geq 1$ if and only if
\begin{equation} \label{pol geq 0}
0 < \lm_2 k\, =
E_{22} k^4 - \lm_1 k^2 + g + \frac{(g\rho)^2}{\lm_1 - E_{11}}\,
= \frac{p_k(\lm_1)}{E_{11} - \lm_1}
\quad \text{and} \quad \lm_1 > 0,
\end{equation}
where
\[
p_k(X) := k^2 X^2 - \big(E_{22}k^4 + E_{11} k^2 + g\big) \, X
+ E_{11} \big(E_{22} k^4 + g\big) - (g\rho)^2.
\]
The discriminant of $p_k$ is
\[
\D(p_k) = \big(E_{11} k^2 - E_{22} k^4 - g\big)^2
+ \big(2g\rho k\big)^2 \,> 0,
\]
and, since $p_k(E_{11})= -(g\rho)^2$, its roots
\[
X_k^{\pm} := \frac{E_{11} k^2 + E_{22} k^4 + g \pm \sqrt{\D(p_k)}}{2k^2}
\]
satisfy
\[
X_k^- < E_{11} < X_k^+ .
\]
Moreover $X_k^- > 0$ if and only if $E_{11} (E_{22} k^4 + g) > g^2 \rho^2$.
It follows that \eqref{pol geq 0} holds if and only if
\[
\lm_1 \in (0,X_k^-) \cup (E_{11},X_k^+), \quad
\lm_2 = f_k(\lm_1),
\]
where $(0,X_k^-)$ is meant to be empty if $X_k^- \leq 0$, and
\[
f_k(\lm_1) := E_{22} \,k^3 - \lm_1 k
+ \frac{1}{k} \Big( g + \frac{(g\rho)^2}{\lm_1-E_{11}}\Big).
\]
$X_k^- \to E_{11}$ and $X_k^+ \to +\infty$ as $k \to +\infty$.
Thus, for every $\lm_1 \neq E_{11}$, there exists an integer $\bar k = \bar k(\lm_1)$ such that $\lm_1 \in (0,X_k^-) \cup (E_{11},X_k^+)$
for all $k \geq \bar k$.
Therefore, for every $k \geq \bar k$,
\eqref{integer eq} holds with $\lm_2 = f_k(\lm_1)$, and we have proved the following lemma.

\begin{lemma} \label{lemma:exists one}
Let $g,\rho,E_{11}$ and $E_{22}$ be fixed, positive constants.
For every fixed $\lm_1 \neq E_{11}$, the parameters $\lm_2$ for which the linearized operator $d_w F(0,\lm_1,\lm_2)$ has a non-trivial kernel
form a sequence $\{\lm_2^{(k)} = f_k(\lm_1): k \geq \bar k(\lm_1)\}$, with
\[
\lm_2^{(k)} = f_k(\lm_1) \to +\infty \quad \text{as} \ k \to \infty.
\]
\end{lemma}

Thus, for every $g, \rho, E_{11}, E_{22}>0$ there exists a set $\mA$ formed by infinitely many curves $\mA_k$ in the parameter quadrant $\{ (\lm_1, \lm_2) : \lm_1 > 0$, $\lm_2>0\}$,
\[
\mA = \bigcup_{k \in \N} \mA_k,
\quad
\mA_k = \big\{ (\lm_1, \lm_2) : \lm_2 = f_k(\lm_1) \big\}
\cap \big\{\lm_1 >0, \ \lm_2 >0\big\}
\]
(see figure \ref{fig 2}), such that the kernel of the linearized operator $d_w F(0,\lm)$ is nontrivial if and only if $\lm \in \mA$.
Note that there is no restriction on $\lm_1$ except that $\lm_1\neq E_{11}$.

\begin{figure}[p]
\begin{center}
\includegraphics
{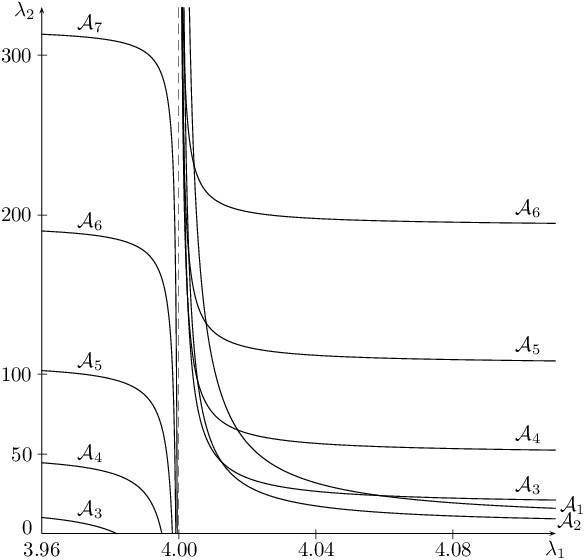}

\vspace{12mm}

\quad \includegraphics[scale=0.93]
{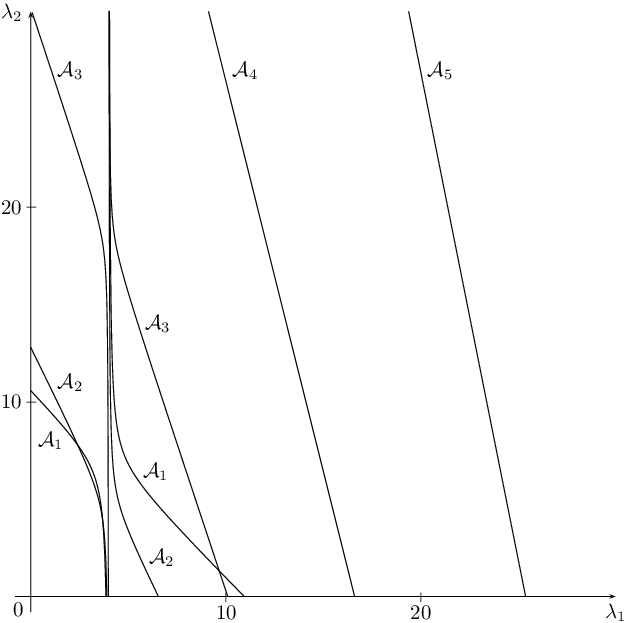}
\end{center}
\caption{Plots of the curves $\mA_k$, $k=1,\ldots,7$,
when $g=9.81$, $g\rho=1$, $E_{11}=4$ and $E_{22}=1$,
first in the region $3.96 < \lambda_1 < 4.10$ and $0 < \lambda_2 < 330$,
with two different scales for the two axes, and
then in the region $0 < \lambda_1, \lambda_2 < 30$, with the same scale
for $\lambda_1$ and $\lambda_2$.}
\label{fig 2}
\end{figure}

\subsection{Double eigenvalues}
The kernel of $d_w F(0,\lm)$ is two-dimensional if and only if
\eqref{integer eq} has two positive integer solutions $k \neq l$,
namely the curves $\mA_k$ and $\mA_l$ cross at $\lm$.
Now, $k \neq l$ solve \eqref{integer eq} if and only if
\begin{equation} \label{kl inter}
\lm_2 = h_{k,l}(\lm_1) := (k+l) \big(E_{22}(k^2+l^2) - \lm_1\big)
\quad \text{and} \quad
q_{k,l}(\lm_1)=0,
\end{equation}
where
\begin{align*}
q_{k,l}(X) := {} & kl X^2
- \big( E_{11} kl + E_{22} kl (k^2 + kl + l^2) - g \big) X \\
& + E_{11} E_{22} kl (k^2 + kl + l^2) - E_{11} g + (g\rho)^2 . \notag
\end{align*}
For all $kl$ sufficiently large,
the discriminant of $q_{k,l}$,
\[
\D(q_{k,l}) = \big( E_{11} kl - E_{22} kl (k^2 + kl + l^2) + g \big)^2
- 4 kl (g\rho)^2,
\]
is positive, and the roots $X_{k,l}^-$ and $X_{k,l}^+$ of $q_{k,l}$ are both greater than $E_{11}$.

Since $h_{k,l}(X_{k,l}^+) < 0$ for all $kl$ sufficiently large,
there are at most finitely many solutions $\lm$ of \eqref{kl inter} with $\lm_1 = X_{k,l}^+$ and $\lm_2>0$.
On the other hand, $h_{k,l}(X_{k,l}^-) > 0$ for all $kl$ sufficiently large, and
\[
X_{k,l}^- \to E_{11}, \quad
h_{k,l}(X_{k,l}^-) \to +\infty
\]
as $kl \to +\infty$.
Thus, we have proved the following lemma.

\begin{lemma} \label{lemma:kl}
Let $g,\rho,E_{11}$ and $E_{22}$ be fixed, positive constants.
The parameters $\lm$ for which the linearized operator $d_w F(0,\lm)$ has a two-dimensional kernel form a sequence $\lm^{(n)} = (\lm^{(n)}_1, \lm^{(n)}_2)$, with
\[
\lm^{(n)}_1 \searrow E_{11}, \quad
\lm^{(n)}_2 \to + \infty \quad \text{as}\ n \to \infty.
\]
\end{lemma}

\begin{remark*}
Double eigenvalues with $\lm_1 < E_{11}$ are possible, provided we assume some additional hypotheses on $E$, namely
\[
E_{22} (kl)^2 < g < E_{22} kl (k^2 + kl + l^2)
\]
and $E_{11}$ sufficiently large.
In any case, they are at most finitely many.
\qed \end{remark*}

\section{Lyapunov-Schmidt reduction}We turn now to study the bifurcation of solutions of \eqref{eq:F=0}.
 Recall that throughout we are dealing with $2\pi$-periodic functions $w$ of zero mean.
Suppose that $\lm^* = (\lm_1^*, \lm_2^*) \in \mA$, $\lm_1^* \neq E_{11}$.
Then the kernel,
\[
V := \Ker d_wF(0,\lm^*) \subset H^3_0,
\]
of the linearized operator,
is a subspace of dimension 1 or 2, depending on the number of integer solutions of \eqref{integer eq}, and the range
\[
R := \Range d_wF(0,\lm^*) \subset H^1_0
\]
is orthogonal to $V$ with respect to the $L^2(0,2\p)$ scalar product, namely
\[
H^1_0 = V \oplus R, \quad
H^3_0 = V \oplus (R \cap H^3_0).
\]
In fact, it is evident from \eqref{dF} that $d_w F(0,\lm)$ is a diagonal operator with respect to the basis of even $2\p$-periodic functions $\{ \cos (j\t) \colon j=1,2,\ldots \}$, for all $\lm$.

Following the classical Lyapunov-Schmidt decomposition, we write
\[
w = v + y, \quad v \in V, \quad y \in R \cap H^3_0,
\]
and denote $\Pi_V, \Pi_R$ the projection onto $V$ and $R$ respectively.
The equation $F(w,\lm)=0$ is then equivalent to the system
\begin{align} \label{Lyap-Schmidt}
\begin{cases}
\Pi_V \,F(v+y, \lm) = 0 & \text{(bifurcation equation)},\\
\Pi_R \,F(v+y, \lm) = 0 & \text{(auxiliary equation)}.
\end{cases}
\end{align}

\begin{lemma}[Auxiliary equation] \label{lemma:aux}
There is a neighbourhood $\mU$ of $(0,\lm^*)$ in $V \times \R^2$, a neighbourhood $U$ of $0$ in $R \cap H^3_0$ and a function $\oy \in C^\infty(\mU, U)$ such that
\[\Pi_R \,F(v + \oy(v,\lm), \lm) = 0
\]
for all $(v,\lm) \in \mU$, and, if $\Pi_R \,F(v + y,\lm) = 0$ with $y \in U$ and $(v,\lm) \in \mU$, then $y = \oy(v,\lm)$.
Moreover, for all $(0,\lm) \in \mU$,
\begin{equation}  \label{oy}
\oy(0,\lm)=0, \quad d_v \oy(0,\lm)=0, \quad
d_{\lm_i} \oy(0,\lm) = 0, \quad i=1,2,
\end{equation}
and there exists a constant $C>0$ such that
\[
\| \oy(v,\lm) \|_{H^3} \leq C \| v \|_{H^3}^2
\]
for all $(v,\lm) \in \mU$.
\end{lemma}

\begin{proof}
Apply the implicit function theorem, and note that, for all $\lm$,
$d_w F(0,\lm)$ is diagonal in the basis $\{ \cos j\t \}$, $j=1,2,\ldots$,
therefore $\Pi_R \, d_w F(0,\lm)v = 0$ for all $v \in V$, for all $\lm$.
\end{proof}

In this way, the bifurcation problem for the equation \eqref{eq:F=0} has been reduced to
\begin{equation} \label{bif eq}
\Pi_V \,F(v + \oy(v,\lm), \lm) = 0,
\end{equation}
with $(v,\lm) \in \mU \subset V \times \R^2$.

\section{Sheets bifurcating from a simple eigenvalue} \label{simple}

Here we study the elementary case in which equation \eqref{integer eq}, and hence the linearized operator \eqref{dF}, has a 1-dimensional kernel.
From the discussions in Section \ref{sec:linearized} this is the case for parameter values $(\lm_1^*,\lm_2^*)$ on the union of countably many curves with countably many points removed.

Hence suppose that there exists a unique integer $k \geq 1$ that satisfies \eqref{integer eq} for $\lm^* = (\lm^*_1, \lm^*_2)$.
Then
\[
V := \Ker d_wF(0,\lm^*) = \big\{ t \cos (k \t) \colon t \in \R \big\}
\]
and, by \eqref{bif eq}, the system \eqref{Lyap-Schmidt} is equivalent to the problem
\begin{equation} \label{biff}
\Phi(t,\lm) :=
\Pi_V \,F \big( t \cos(k\t) + \oy(t \cos(k\t),\lm),\, \lm \big) = 0.
\end{equation}
$\Phi$ is a smooth real valued map of three real variables $(t,\lm_1,\lm_2)$, defined on a neighbourhood of $(0,\lm_1^*,\lm_2^*)$.
Since $F(0,\lm)=0$ for all $\lm$, from Lemma \ref{lemma:aux} it follows that
\[
\Phi(0,\lm)=0 \quad \text{for all } (0,\lm) \in \mU.
\]
Also, using \eqref{oy} and the orthogonality of $V$ and $R$,
\begin{equation} \label{da Phi=0}
\partial_t \Phi(0,\lm^*) = 0.
\end{equation}
To find solutions of \eqref{biff} with $t \neq 0$ we invoke the implicit function theorem in the usual analytic approach to bifurcation problems.

\begin{theorem} \label{thm:simple bif}
Suppose that there exists a unique integer $k \geq 1$ that satisfies \eqref{integer eq} for $\lm^* = (\lm^*_1, \lm^*_2)$, with $\lm_1^* \neq E_{11}$.
Then there exist neighbourhoods $\mU_1$ of $(0,\lm_1^*)$ in $\R^2$ and $U_1$ of $\lm_2^*$ in $\R$,
and a map $\olm_2 \in C^\infty(\mU_1, U_1)$,
with $\olm_2(0,\lm_1^*)=\lm_2^*$, such that
\[
\Phi(t,\lm_1,\olm_2(t,\lm_1)) = 0 \quad
\text{for all } (t,\lm_1) \in \mU_1,
\]
and, if $\Phi(t,\lm_1,\lm_2)=0$, with $(t,\lm_1) \in \mU_1$, $t \neq 0$ and $\lm_2 \in U_1$,
then $\lm_2 = \olm_2(t,\lm_1)$.
As a consequence,
\[
F \big( w(t,\lm_1), \lm_1, \olm_2(t,\lm_1) \big)=0,
\]
where
\[
w(t,\lm_1) := t \cos(k\t) + \oy \big( t\cos(k\t), \lm_1, \olm_2(t,\lm_1)\big) = t \cos(k\t) + O(t^2).
\]
\end{theorem}

\begin{proof}
First, we prove that
\begin{equation} \label{transversality}
\partial^2_{t,\lm_2} \Phi (0,\lm^*) \neq 0.
\end{equation}
By \eqref{oy},
\[
\partial_t \Phi(0,\lm) = \Pi_V \, d_w F(0,\lm)\,
\big( 1+ d_v \oy(0,\lm) \big) \cos(k \t) = \Pi_V \, d_w F(0,\lm)\,
 \cos(k \t),
\]
and \[
\partial_{t,\lm_2}^2 \Phi(0,\lm^*) =
\Pi_V \, d^2_{w,\lm_2} F(0,\lm^*)\,\cos(k\t).
\]
By \eqref{dF},
\[
d_w F(0,\lm) \cos(k\t) =
\Big( - k^2 E_{22} + \lm_1 + \frac{\lm_2}{k}
 - \Big( g + \frac{(g\rho)^2}{\lm_1 - E_{11}} \Big)\,\frac{1}{k^2}\,
\Big) \cos(k\t),
\]
whence
\[
\Pi_V \, d^2_{w,\lm_2} F(0,\lm^*) \cos(k\t)
= \frac{1}{k^*}\, > 0,
\]
and \eqref{transversality} is proved.
Since
\[
\Phi(t,\lm) = \int_0^t (\partial_t \Phi) (x,\lm)\,dx \,
= \,t \int_0^1 (\partial_t \Phi) (zt,\lm)\,dz,
\]
it follows that $\Phi(t,\lm) =0$, with $t \neq 0$,
if and only if $\ph(t,\lm)=0$, where
\[
\ph(t,\lm) := \int_0^1 (\partial_t \Phi)(xt,\lm)\,dx.
\]
From the smoothness of $\Phi$ it follows that $\ph$ is also smooth.
By \eqref{da Phi=0}, $\ph(0,\lm^*)=0$.
Moreover, $\partial_{\lm_2}\ph(0,\lm^*) \neq 0$ by \eqref{transversality}.
The result now follows from the implicit function theorem.
\end{proof}

\begin{remark*} Since
\[
\Pi_V \, d^2_{w,\lm_1} F(0,\lm^*) \cos(k\t)
= 1 + \Big( \frac{g\rho}{k(\lm_1^* - E_{11})} \Big)^2 > 0,
\]
the role of $\lm_1$ and $\lm_2$ can be swapped.
\qed \end{remark*}

\section{Bifurcation from a double eigenvalue} \label{double}

We have observed that for any $(\lm_1,\lm_2)$ there are at most two positive integer solutions, $k,\,l,$ of \eqref{integer eq}, and this happens only if $f_k(\lm_1) = f_l(\lm_1) = \lm_2$.
Suppose that there are indeed two such solutions, $k$ and $l$,
with
\begin{equation} \label{non-res}
\frac{\max\{k,l\} }{\min\{k,l\} }\, \notin \Z.
\end{equation}
Let $Z_k$ be the closure of $\Span \{ \cos (jk\t): j \in \N\}$ in $L^2(0,2\p)$, and similarly for $Z_l$.
Now note that if one seeks waves with minimal period $2\pi/k$ or $2\pi/l$, the original bifurcation problem \eqref{eq:F=0} may be specialized to a problem on $Z_k$ or $Z_k$ and the reduced problem \eqref{bif eq} is similarly restricted to $Z_k$ or $Z_l$.
In each of these restricted settings separately, only one solution, $k$ or $l$, of
\eqref{integer eq} is relevant, and
 there is a simple eigenvalue from which a curve of solutions in $Z_k\cap H_0^3$ or $Z_l\cap H_0^3$ bifurcates, exactly as in the preceding section.
However, we will now show that other solutions that are neither in $Z_k$ nor $Z_l$ bifurcate at $\lm^*$ when $k$ and $l$ are solutions of \eqref{integer eq} with $\lm= \lm^*$ and \eqref{non-res} holds.

In this case the kernel of the linearized problem is two-dimensional,
\[
V := \Ker d_wF(0,\lm^*) = \big\{ t_1\cos (k \t) + t_2 \cos (l\t) :\, (t_1,t_2) \in \R^2 \big\},
\]
and the bifurcation problem \eqref{bif eq} is
\begin{equation} \label{Phi=0}
\Phi (t_1,t_2,\lm) = 0, \quad \lm = (\lm_1,\lm_2),
\end{equation}
where
\[
\Phi (t_1,t_2,\lm) := \Pi_V F( v + \oy(v,\lm), \lm),
\quad v = t_1 \cos(k\t) + t_2 \cos (l\t).
\]
Let
$\Phi_k \cos (k\t) := \Pi_k \Phi$ and $\Phi_l \cos (l\t):= \Pi_l \Phi
$, where $\Pi_k$ and $\Pi_l$ denote the projections onto $\Span \{ \cos(k\t) \}$ and $\Span \{ \cos(l\t) \}$, respectively. Thus
\eqref{Phi=0} becomes
\begin{align*}
\Phi_k (t_1,t_2,\lm_1,\lm_2) & = 0, \\
\Phi_l (t_1,t_2,\lm_1,\lm_2) & = 0,
\end{align*}
a system of two equations in four unknowns which is satisfied by $(0,0,\lm_1,\lm_2)$ for all $\lm$.
The key to our result is the following observation.

Suppose that $t_1=0$, and $v=t_2 \cos(l\t)$, $t_2 \in \R$.
Then an application of Lemma \ref{lemma:aux} in the subspace $Z_l$ of $2\p/l$-periodic functions yields that $\oy(v,\lm) \in Z_l \cap R$, because of the local uniqueness in the implicit function theorem.
Hence $v+\oy(v,\lm)$ is $2\p/l$-periodic, therefore $F(v+\oy(v,\lm), \lm)$ is also $2\p/l$-periodic.
As a consequence,
\begin{equation} \label{inv observ k}
\Phi_k(0,t_2,\lm) = 0 \ \text{ for all } t_2,\lm.
\end{equation}
For the same reason,
\begin{equation} \label{inv observ l}
\Phi_l(t_1,0,\lm) = 0 \ \text{ for all } t_1,\lm.
\end{equation}
We now require the non-degeneracy condition
\begin{equation} \label{nondeg}
\Big( \frac{g\rho}{\lm_1^* - E_{11}} \Big)^2 \neq kl.
\end{equation}

\begin{remark*}
It is easily checked that condition \eqref{nondeg} is equivalent to the geometrical assumption that $f_k'(\lm_1^*) \neq f_l'(\lm_1^*)$.
In other words, the curves $\mA_k$ and $\mA_l$ are not tangential at their intersection point $\lm_1^*$.
\qed \end{remark*}

\begin{theorem} \label{thm:secondary bif}
Suppose that there exist two integers $k,l$ that satisfy \eqref{integer eq} for $\lm^* = (\lm^*_1, \lm^*_2)$ where $\lm_1^* \neq E_{11}$, and \eqref{non-res} and \eqref{nondeg} hold.
Then there exist neighbourhoods $\mU_2$ of the origin and $U_2$ of $\lm^*$ in $\R^2$,
and functions $\olm(t_1,t_2) = (\olm_1(t_1,t_2), \olm_2(t_1,t_2))$, $\olm \in C^\infty(\mU_2, U_2)$, with $\olm(0,0) = \lm^*$, such that
\[
\Phi(t_1,t_2,\olm(t_1,t_2)) = 0 \quad \text{for all } (t_1,t_2) \in \mU_2,
\]
and, if $\Phi(t_1,t_2,\lm)=0$ with $(t_1,t_2) \in \mU_2\setminus \{(0,0)\}$ and $\lm \in U_2$, then $\lm = \olm(t_1,t_2)$.
As a consequence,
\[
F \big( w(t_1,t_2), \olm(t_1,t_2) \big)=0,
\]
where
\begin{align*}
w(t_1,t_2) := {} & \, t_1 \cos(k\t) + t_2 \cos(l\t) + \oy \big( t_1\cos(k\t) + t_2 \cos(l\t),\, \olm(t_1,t_2) \big) \\
= {} & \, t_1 \cos(k\t) + t_2 \cos(l\t) + O(t_1^2 + t_2^2).
\end{align*}
\end{theorem}

\begin{proof}
Let $\Psi := (\Psi_k, \Psi_l)$,
\[
\Psi_k(t_1,t_2,\lm) := \int_0^1 (\partial_{t_1} \Phi_k) (xt_1,t_2,\lm)\,dx,
\quad \
\Psi_l(t_1,t_2,\lm) := \int_0^1 (\partial_{t_2} \Phi_k) (t_1,xt_2,\lm)\,dx.
\]
$\Psi_k$ and $\Psi_l$ are smooth by \eqref{inv observ k} and \eqref{inv observ l}.
Moreover, since
\begin{align} \label{good k}
\Psi_k(0,0,\lm) & = \Pi_k \, d_w F(0,\lm) \cos (k\t) \\
& = - E_{22} k^2 + \lm_1 + \frac{\lm_2}{k}\,
- \Big( g + \frac{(g\rho)^2}{\lm_1 - E_{11}} \Big) \frac{1}{k^2}, \notag
\end{align}
and, analogously,
\begin{equation} \label{good l}
\Psi_l(0,0,\lm) = - E_{22} l^2 + \lm_1 + \frac{\lm_2}{l}\,
- \Big( g + \frac{(g\rho)^2}{\lm_1 - E_{11}} \Big) \frac{1}{l^2},
\end{equation}
it follows that
\[
\Psi(0,0,\lm^*) = 0.
\]
To apply the implicit function theorem to $\Psi$ at the point $(0,0,\lm^*)$, it is sufficient to prove that the 2$\times$2 matrix representing the linear map $\partial_\lm \Psi(0,0,\lm^*)$ is invertible.
Now, differentiating \eqref{good k} and \eqref{good l} with respect to $\lm_1$ and $\lm_2$,
\[
\det \big(\partial_\lm \Psi(0,0,\lm^*) \big) =
\Big\{ \Big( \frac{g\rho}{\lm_1^* - E_{11}} \Big)^2 - kl \Big\}
\Big( \frac1k - \frac1l \Big) \frac1{kl},
\]
which is nonzero by \eqref{nondeg}.
\end{proof}

\begin{remark*}
By the definition of $\Psi$, for $(t_1,t_2) \in \mU_2$, with $t_1\neq0$ and $t_2\neq0$, Theorem \ref{thm:secondary bif} gives solutions of problem \eqref{Phi=0} which do not belong to $Z_k$ nor $Z_l$, as it had been stated above.
\qed \end{remark*}

\textbf{Acknowledgements.}
JFT acknowledges the support of a Royal Society/Wolfson Research Merit Award.
PB is supported by the European Research Council under FP7
and the Italian PRIN \emph{Variational methods and nonlinear differential equations}.
The main part of this paper has been written when PB was supported by the UK EPSRC at the University of Bath.
The authors are grateful to M.C.W. Jones for the reference to the work of Bohr mentioned in the Introduction.
\small{

}

\begin{thebibliography}{99}

\bibitem{Antman} Antman, S.S., \emph{Nonlinear Problems of Elasticity.}
Springer-Verlag, New York, 1995.

\bibitem{Baldi-Toland-var} Baldi, P., and Toland, J.F.,
Steady periodic water waves under nonlinear elastic membranes,
\emph{preprint}.

\bibitem{bohr} Bohr, N., {\em Collected Works}, North-Holland, Amsterdam, (1906), page 67-78.

\bibitem{cd} Christodoulides, P., and Dias, F.,
Resonant capillary-gravity interfacial waves,  \emph{J. Fluid Mech.} \textbf{265} (1994), 303--343.

\bibitem{JT} Jones, M.C.W., and Toland, J.F.,
Symmetry and the bifurcation of capillary-gravity waves,
\emph{Arch. Rational Mech. Anal.} \textbf{96} (1) (1986), 29--53.

\bibitem{strain} Plotnikov, P.I., and Toland, J.F.,
Strain-gradient theory of hydroelastic travelling waves and their
singular limits.
University of Bath, in preparation.


\bibitem{Toland-heavy} Toland, J.F.,
Heavy hydroelastic travelling waves,
\emph{Proc. R. Soc. Lond.} A \textbf{463} (2007), 2371--2397.

\bibitem{TJ} Toland, J.F., and Jones, M.C.W.,
The bifurcation and secondary bifurcation of capillary-gravity waves. \emph{Proc. R. Soc. Lond.} A \textbf{399} (1985), 391--417.

\bibitem{wilton} Wilton, J.R.,
On ripples.
\emph{Phil. Mag.} Ser. 6, \textbf{29} (1915), 688--700.

\bibitem{Zyg} Zygmund, A.,
\emph{Trigonometric Series I \& II}.
Corrected reprint (1968) of 2nd edn., Cambridge University Press, Cambridge, 1959.

\end{thebibliography}
\end{document}